\def\seq#1_#2{\langle #1_#2:#2\in\omega\rangle}
\def\fc#1|#2{#1\uparrow#2}
\def\ain{\subseteq^*}

\def\N{{\mathbb N}}
\def\Q{{\mathbb Q}}
\def\R{{\mathbb R}}

\def\set#1:#2.{{\{\,#1:#2\,\}}}
\def\Int{\mathop{\rm Int}}

\def\so{\mathop{\rm so}}

\documentclass{article}
\title{On large sequential groups}
\author{Alexander Y.~Shibakov\footnote{Tennessee Tech.\ Universiy,
email: {\tt ashibakov@tntech.edu}}}
\usepackage{amssymb}
\usepackage{amsthm}
\usepackage{amsmath}

\newtheorem{proposition}{Proposition}
\newtheorem{lemma}{Lemma}
\newtheorem{example}{Example}

\newtheorem{question}{Question}
\theoremstyle{definition}
\newtheorem{remark}{Remark}

\newtheorem{claim}{Claim}
\begin{document}
\maketitle
\renewcommand{\thefootnote}{}



\renewcommand{\thefootnote}{\arabic{footnote}}
\setcounter{footnote}{0}
\begin{abstract}
\noindent We construct, using $\diamondsuit$, an example of a sequential group
$G$ such that the only countable sequential subgroups of $G$ are
closed and discrete, and the only quotients of $G$ that have a
countable pseudocharacter are countable and Fr\'echet. We also show
how to construct such a $G$ with several additional properties (such as make $G^2$
sequential, and arrange for every sequential subgroup of $G$ to be
closed and contain a nonmetrizable compact subspace, etc.).

Several results about $k_\omega$ sequential groups are proved. In
particular, we show that each such group is either locally compact and
metrizable or contains a closed copy of the sequential fan. It is also
proved that a dense proper subgroup of a non Fr\'echet $k_\omega$ sequential
group is not sequential extending a similar observation of T.~Banakh about
countable $k_\omega$ groups.
\end{abstract}
\section{Introduction}\label{introduction}
The study of convergence in the presence of a topologically compatible
algebraic structure dates back to the famous Birkhoff-Kakutani theorem
on the metrizability of first-countable groups. A number of authors
have studied properties such as sequentiality and Fr\'echetness in topological
groups since then, obtaining a variety of results on metrizability of
such groups, as well as discovering several pathologies exhibited by
these classes of spaces.

Most of these efforts have been concentrated on studying countable
sequential and Fr\'echet groups. One possible reason is the intuitive
idea that convergence phenomena are primarily countable in
nature. The Fr\'echet property is also inherited by arbitrary
subspaces essentially reducing the study of separable Fr\'echet groups
to that of countable ones. This research culminated in a beautiful result
of Hru\v s\'ak and Ramos-Garc\'\i a establishing the independence of
the metrizability of all countable Fr\'echet groups from the axioms of
ZFC (their celebrated solution of the well known Malykhin's problem,
see \cite{HRG}).

A number of results exist on countable sequential groups, as
well. The class of $k_\omega$ countable groups (see below for the
definitions of all the concepts used in this introduction) is well understood. The paper
by I.~Protasov and E.~Zelenyuk, \cite{PZ1} introduced a number of
techniques for the study of such groups, that have since then been
applied by a host of researchers to the study of not only countable
sequential but also precompact, complete, etc. group topologies
(see \cite{TB2}, \cite{BDMW}, and \cite{MPS}, for example). An elegant result by
E.~Zelenyuk, \cite{Z} provides a full topological classification of general
countable $k_\omega$ groups.

Not much is known about countable
sequential groups that are not $k_\omega$, aside from some consistent
examples of groups with various pathologies. In particular, there are no
known ZFC examples of sequential countable groups that are not $k_\omega$.

Even less is known about the general, uncountable sequential
groups. It is known that all compact sequential groups are metrizable
(a corollary of a result by Shapirovskii, \cite{SHA}). A
$\Sigma$-product of uncountably many unit circles provides a
well-known example of a Fr\'echet countably compact group that is not
metrizable. Naturally, all separable subspaces of this group are
metrizable, and thus Fr\'echet. It is unknown at the present time if a
countably compact topological group can be made sequential and not
Fr\'echet.

In the class of sequential groups another property comes
into play: in contrast to the Fr\'echet property, sequentiality is not
hereditary. It is an easy observation 
that spaces that are hereditarily sequential are
Fr\'echet. On the other hand, every non discrete sequential space has a non
discrete sequential subspace (a convergent sequence) and every
separable nonmetrizable Fr\'echet
group has a non nonmetrizable countable Fr\'echet subgroup.

In this paper we attempt to demonstrate that in the realm of the
sequential groups `separable' is not as easily replaced with
`countable'. We build, using $\diamondsuit$, an example of an
uncountable (separable if desired) sequential group $G$ all of whose
countable sequential subgroups are closed and discrete (in fact, all of its
sequential subgroups are closed).

One might hope that an attempt to reduce the size of a sequential
group that is not Fr\'echet by taking quotients might produce a
`smaller' group somehow exhibiting the non Fr\'echet property of the `big'
group. A rather weak condition that is desirable in a sequential group
is countable pseudocharacter. Note that a sequential group can contain
arbitrary sequential compact subspaces, while having a countable psudocharacter
would force all such compact subspaces to be first countable. This is
indeed possible if the group is $k_\omega$ (see Lemma~\ref{Sresolve}). The
example above, however, has an additional property that every quotient
of countable psudocharacter of $G$ is countable and Fr\'echet.

Our example is therefore a sequential non Fr\'echet group that does
not `reflect' its pathologies to countable subgroups or smaller
quotients. In the process of building the example, we prove a number
of results about general (not necessarily countable) sequential
$k_\omega$ groups that may be of independent interest.

\section{Definitions and terminology}\label{defterm}
We use standard set theoretic notation and terminology,
see \cite{KU}. If $X$ is a topological space and $C\subseteq X$, $x\in
X$, we write $C\to x$ to indicate that $C$ converges to $x$, i.e.\
$C\ain U$ for every open $U\ni x$.
Recall that a space $X$ is called {\it sequential\/} if for every $A\subseteq X$
such that $\overline{A}\neq A$ there is a $C\subseteq A$ such that
$C\to x\not\in A$.

Let $A\subseteq X$. Define the {\it sequential closure of $A$}, $[A]'
= \{\,x\in X: C\to x\text{ for some }C\subseteq A\,\}$.
Put $[A]_\alpha = \cup\{[A_\beta]':\beta<\alpha\}$ for $\alpha\leq\omega_1$.

Define $\so(X) = \min\{\,\alpha \leq \omega_1 : [A]_\alpha =
\overline{A}\text{ for every }A\subseteq X\,\}$. $\so(X)$ is called
the {\it sequential order\/} of $X$. Spaces of sequential order $\leq
1$ are called Fr\'echet.

The {\it psudocharacter $\psi(x,X)$ of
$X$ at $x\in X$\/} is the smallest cardinality of a family ${\cal U}$
of open neighborhoods of $x$ such that $\cap{\cal U}=\{x\}$. When $G$
is a topological group we write simply $\psi(G)$. It is well known
(and easy to show) that the character (i.e.\ the smallest cardinality
of a base of open neighborhoods at a point) and the pseudocharacter
coinside for compact $X$.

Given an arbitrary set $X$ and a collection of subsets ${\cal
K}\subseteq2^X$, each of which is equipped with a topology $\tau_K$
where $K\in{\cal K}$ one can introduce a topology on $X$ as follows: a
subset $U\subseteq X$ is open in $\tau$ if and only if each $U\cap K$
is open in $\tau_K$. We will say that $\tau$ is {\it determined\/} by
${\cal K}$. Note that as a subspace of $X$ in this topology, the
topology of $K\in{\cal K}$ may be different from $\tau_K$. On the
other hand, if each $(K,\tau_K)$ is compact and Hausdorff and for any
$K,F\in{\cal K}$ the intersection $K\cap F$ is closed in each $K$ and
$F$ and inherits the same topology from either space then the topology of each
$K\in{\cal K}$ as a subspace of $X$ is exactly $\tau_K$. Note that
$\tau$ is $T_1$ if and only if each $\tau_K$ is.

If the family ${\cal K}$ above is countable, consists of compact
spaces and satisfies the condition mentioned at the end of the
previous paragraph, $X$ will be called a {\it $k_\omega$ space}.

For general groups $1$, $\cdot$, and ${}^{-1}$ stand for the unit,
the group operation and the algebraic inverse, respectively. When the
group is known to be abelian these change to the traditional $0$, $+$,
and $-$. It will be convenient to use $\sum^kB$ for the sum
$B+\cdots+B$ of $k$ copies of some subset $B$ of an abelian group
$G$. Put $\sum^0B=B^0=\{0\}$ (here $B^k$ is a Cartesian product of
$B$'s in general).

Call a family ${\cal U}$ of subsets of some group $G$ {\it stable\/}
if for any $U,V\in{\cal U}$ there exists a $W\in{\cal 
U}$ such that $WW^{-1}\subseteq U\cap V$, and $HU=U$ for each
$U\in{\cal U}$ where $H=\cap{\cal U}$ is a normal subgroup of $G$.

If $G$ is a topological group and ${\cal U}$ is a stable
family of open subsets of $G$ it is easy to see that
$H=\cap{\cal U}$ is a closed subgroup of $G$ and $\set p(U):U\in{\cal
U}.$ forms a base of open neighborhoods of $0$ in $G/H$ in some
Hausdorff topology coarser than the quotient topology induced by the
natural quotient map $p:G\to G/H$.

If, in addition, $G$ is an abelian group and $V\subseteq G$ is an open
neighborhood of $0$ then one can extend ${\cal U}$ to some stable
family ${\cal U}'$ of open subsets of $G$ such that there
is a $V'\in{\cal U}'$ with the property that $V'\subseteq V$ and the
cardinalities of ${\cal U}$ and ${\cal U}'$ are the same.

In the case of a non-abelian group $G$, the construction above has to
guarantee that $H$ is a normal subgroup of $H$. This can be done in
the case of a separable $G$ by ensuring that for any $U\in{\cal U}$
and any $x$ in some countable dense subset of $G$ there exists a
$V\in{\cal U}$ such that $xVx^{-1}\subseteq U$.

The simple lemma below demonstrates the main application of stable families.
\begin{lemma}\label{stabres}
Let $G$ be a topological group, ${\cal U}$ be a stable family of open
subsets of $G$, $E\subseteq G$ be a countable set, and
${\cal K}$ be a countable family of subsets of $G$. If $G$ is either
abelian or separable, one can extend ${\cal U}$ to a stable family
${\cal U}'$ of open subsets of $G$ such that $|{\cal U}|=|{\cal U}'|$,
$p(\overline{K}\cap
E)=\overline{p(K)}\cap p(E)$ for each $K\in{\cal K}$, and the
restriction of $p$ to $E$ is one-to-one. Here $p:G\to G/H$ is the
natural quotient map and $H=\cap{\cal U}'$ is a closed normal subgroup
of $G$.
\end{lemma}
\begin{proof}
Extend ${\cal K}$ if necessary to guarantee that every singleton from
$E$ is in ${\cal K}$. For each pair $(e,K)\in E\times{\cal K}$ such
that $e\not\in\overline{K}$ find an open neighborhood of unity
$U\subseteq G$ such that $e\cdot U\cdot U\cap\overline{K}\cdot U\cdot
U=\varnothing$. Using the remarks preceding the statement of the
lemma, extend ${\cal U}$ to a stable family ${\cal U}'$ such that for
each open set $U$ constructed above there is a $V\in{\cal U}'$ such
that $V\subseteq U$.
\end{proof}

Note that extending a stable family ${\cal U}$ to a stable family
${\cal U}'$ of open subsets of $G$ produces a pair of
open group homomorphisms
$G\smash{\mathop{\rightarrow}\limits^{p'}G/H'\smash{\mathop{\rightarrow}\limits^{p}}}G/H$.
Here $H=\cap{\cal U}$, $H'=\cap{\cal U}'$.

The next lemma is a corollary of a more general result.

\begin{lemma}[see \cite{K}]\label{hquot}
An image of a Fr\'echet space under an open map is Fr\'echet. In
general, open maps do not raise the sequential order of a space.
\end{lemma}

An uncountable group $G$ will be called {\it co-countable\/} if every quotient of
$G$ that has countable pseudocharacter is countable.

Recall that a (necessarily abelian) group $G$ is called {\it boolean\/} if $a+a=0$
for every $a\in G$. Every boolean group can be naturally viewed as as
vector space over the two element field ${\mathbb F}_2$.

\section{Test Spaces}\label{TS}
The study of sequential spaces often makes use of a wide variety of smaller
(usually countable) canonical spaces to investigate various
convergence phenomena. Below we present the definitions of the test
spaces used in this paper.

Let $S_n=[\omega]^{\leq n}$. Put $U\subseteq S_n$ open if and only if
for every $s\in U$ the set $\{\,s^\frown n\in S_n : s^\frown n\not\in
U\,\}$ is finite. Now each $S_n$ is sequential and $\so(S_n)=n$ for
$n<\omega$, whereas $\so(S_\omega)=\omega_1$. $S_2$ is known also as {\em
  Arens' space} while $S_\omega$ is referred to as {\em
  Arkhangel'skii-Franklin space}. The {\it sequential fan $S(\omega)$\/} is defined as
$\omega^2\cup\{0\}$ with the topology in which each point of
$\omega^2$ is isolated and the base of neighborhoods of $0$ is formed
by $U_f=\{0\}\cup\set (i,j):j\geq f(i).$ where $f:\omega\to\omega$.

It is convenient to have a `functional' description of a closed
embedding of $S_2$ into a space which is provided by the following definition.
An injective map $s:\omega^2\to X$ is called a {\it
free $S_2$-embedding\/} into $X$ (with respect to some topology $\tau$
on $X$) if $s(i,j)\to a_i$ and $a_i\to v$ where all $s(i,j)$, $a_i$,
and $v$ are distinct, and each set of the form $D=\set s(i,j):j\leq
f(i).$ for some $f:\omega\to\omega$ is a closed discrete subset of
$X$. If the topology of $X$ is determined by a countable family of
compact subspaces $\seq F_n$ then the last requirement
is equivalent to the condition that each set $\set i:s(i,j)\in
F_n\hbox{ for some $j\in\omega$}.$, $n\in\omega$ is finite. The point
$v$ will be called {\it the vertex\/}
of $s$. It is well known (see \cite{N}) that a free $S_2$-embedding
exists for every countable $k_\omega$ non discrete group and every
countable sequential non Fr\'echet space.

The subspace $S_2^-=\set x\in S_2:|x|\in\{0,2\}.$ of $S_2$ is a
standard example of a non sequential subspace of a sequential space.

The space $D_\omega$ (called ${\mathbb L}$ in \cite{VD}) is defined as
follows. $D_\omega=\omega^2\cup\{(\omega,\omega)\}$ where all the points in
$\omega^2$ are isolated and the base of open neighborhoods of $(\omega,\omega)$
consists of $(\omega\setminus
n)\times\omega\cup\{(\omega,\omega)\}$. The subspace
$(\omega+1)^2\setminus(\{\omega\}\times\omega\cup\omega\times\{\omega\})$
of the product of two convergent sequences is homeomorphic to $D_\omega$.

The diagonal $\Delta=\set
((m,n),\{m,m+n+1\}):m,n\in\omega.\cup\{((\omega,\omega),\varnothing)\subseteq
D_\omega\times S_2$ is a closed subset of the product homeomorphic to
$S_2^-$ (see \cite{TU1}, Remark~5.3). This example is an old result of
van Douwen (see \cite{VD}).

\section{$k_\omega$ groups}\label{kwg}
The following lemma is a group theoretic version of a result of E.~van
Douwen on the non productive nature of sequentiality mentioned at the
end of Section~\ref{TS}. Just as the
original result, its proof embeds a closed copy of $S_2^-$ in the group.

\begin{lemma}[see \cite{BZ1}, Lemma~4]\label{vD}
If $G$ is a topological group that contains closed copies of both
$S(\omega)$ and $D_\omega$ then $G$ is not sequential.
\end{lemma}

The result below is probably folklore. The version presented here does not
require the knowledge of the topology on $G$ in advance. The proof is
given for the abelian case only, although the statement holds in a
more general setting.
\begin{lemma}\label{komega}
Let $G$ be a group, and $\seq F_n$ be a cover of $G$.
Suppose each $F_n$ is given a compact Hausdorff topology such that for
any $i,j\in\omega$ the set $F_i\cap F_j$ is closed in both $F_i$ and
$F_j$ and the induced topologies are the same. Suppose further that
the sums, inverses, and unions of any finite number of $F_i$'s are contained in
some (possibly different) $F_n$'s and that
the addition and algebraic inverse maps restricted to
the corresponding (products of) compacts are continuous (for any large
enough compact range). Then the
topology $\tau$ determined by $\seq F_n$ on $G$ is a Hausdorff group topology.
\end{lemma}
\begin{proof}
Since $\tau$ is easily seen to be $T_1$ it is enough to show that
$\tau$ is a group topology on $G$. Since $\seq F_n$ covers $G$, $\tau$
is invariant with respect to translations.

Let $U$ be a subset of $G$ such that $0\in U$ and $U\cap F_n$ is
open for every $n\in\omega$. Build, by induction on $n\in\omega$, a family
$V_n$, such that:
\begin{itemize}
  \item[(1)] $V_n$ is an open subset of $\cup_{i\leq n}F_i$ such that
    $0\in V_n\subseteq \overline{V_n}\subseteq U$;

  \item[(2)] $\overline{V_n}-\overline{V_n}\subseteq U$

  \item[(3)] $\overline{V_m}\subseteq V_n$ for $m<n$.

\end{itemize}
Assume $F_0=\{0\}$ and suppose $V_i$ have been constructed for
$i<n$. Since $\overline{V_{n-1}}$ is a compact subset of $\cup_{i\leq
  n}F_i$ and $-:F_m^2\to F_k$ is continuous (as well as each
  embedding $F_i\subseteq F_m$),
  where $F_m\supseteq\cup_{i\leq n}F_i$, $F_k\supseteq F_m-F_m$, one
can find an open $V_n\subseteq\cup_{i\leq n}F_i$ such that
$\overline{V_{n-1}}\subseteq V_n$ and
$\overline{V_n}-\overline{V_n}\subseteq U$. Now (1)--(3) are immediate
so put $V=\cup\seq V_n$. Then (1) and (3) imply that $V$ is open in
$\tau$ and (2) and (3) show that $V-V\subseteq U$.
\end{proof}

As a simple application of the lemma above observe that given a
topological group $G$ and a countable family ${\cal K}$ of compact
subspaces of $G$ the finest group topology on $G$ that induces the
original topology on each $K\in{\cal K}$ is $k_\omega$. For a proof
simply consider the `algebraic closure' of ${\cal K}$.

The next lemma is not stated in full generality (see Remark~\ref{Rgroups}) as
we are only interested in its applications in the more narrow setting below.

\begin{lemma}\label{seq}
Let $G$ be a boolean $k_\omega$ group and $D\subseteq G$ be an infinite closed discrete
subset of $G$. Let $a\in G$. Then there exists an infinite subset $C\subseteq D$
such that the finest group topology on $G$ which is coarser
than the original topology on $G$ and such that $C\to a$ is a
$k_\omega$ Hausdorff topology on $G$.
\end{lemma}
\begin{proof}
Let $\seq F_n$ be a collection of compact subspaces of $G$, closed
under finite unions, sums, and intersections that determines the
topology of $G$ (in particular, such a family is always a cover for
$G$). Since translations are homeomorphisms, assume $a=0$ by
translating $D$ if necessary.

Let $D=\seq d_n$. Choose, by induction on $i\in\omega$, a sequence
$\seq n_i$ such that for $c_i=d_{n_i}$, $C_i=\set c_j:j\leq
i.\cup\{0\}$, and $m>i$
$$
(\cup_{j\leq i}F_j+\sum^iC_i)\cap(c_m+\cup_{j\leq
  i}F_j+\sum^iC_i)=\varnothing.
$$
Since $(\cup_{j\leq i}F_j+\sum^iC_i)-(\cup_{j\leq i}F_j+\sum^iC_i)$ is
compact, and $D$ is closed discrete, such
a choice of $n_i$ is possible. For convenience assume $c_0=0\in D$.

Put $C=\seq c_i$, $S_k=\sum^kC$, and consider the natural addition map
$s_{k,n}:C^k\times F_n\to S_k+F_n$ defined as
$s_{k,n}(c^1,\dots,c^k,x)=c^1+\cdots+c^k+x$. Introduce a natural topology on
$C$ such that $C\to0$ and view $s$ as a quotient map onto its image.
Suppose that with such a topology the image, $S_k+F_n$, is Hausdorff for
all $k<K$ and all $n\in\omega$. Choose an arbitrary $n\in\omega$ and
set $i=\max\{K,n\}$. To simplify the notation put $s=s_{K,n}$,
$S=S_K$. Note that the family $(C_i^-)^K$, $E_{j,l}$,
$j,l\leq i$ forms a clopen cover of $C^K$ where $C_i^-=(C\setminus
C_i)\cup\{0\}$ and
$E_{j,l}=C\times\cdots\times\{c_l\}\times\cdots\times C$ where the
$j$-th factor is equal to $\{c_l\}$ while the remaining $K-1$ factors are
$C$. Consider another addition map
$s^+:(C_i^-)^K\times(F_n+\sum^KC_i)\to \sum^KC_i^-+F_n+\sum^KC_i$. Note
that $s(C^K\times F_n)\subseteq
s^+((C_i^-)^K\times(F_n+\sum^KC_i))$. The map $s^+$ can be further written
as $s^+(c^1,\dots,c^K,x)=s_2(s_1(c^1,\dots,c^K),x)$ where both
$s_1:(C_i^-)^K\to \sum^KC_i^-$ and $s_2:(\sum^KC_i^-)\times(F_n+\sum^KC_i)$
are again additions.

Let $c'+x'=c''+x''$ where $c',c''\in\sum^KC_i^-$ and $x',x''\in
F_n+\sum^KC_i$. Choose (using the assumption that $G$ is boolean) some
representations $c'=c_1'+\cdots+c_p'$ and $c''=c_1''+\cdots+c_q''$
where $c_j'=c_{\nu(j)}$ and $c_j''=c_{\mu(j)}$, both $\nu$ and $\mu$ are
strictly increasing and $\nu(\cdot)>i$, $\mu(\cdot)>i$. Put
$r=\max\{p,q\}$. Show that $c'=c''$ by an induction on $r$. If
$\nu(p)\not=\mu(q)$, say $\nu(p)>\mu(q)$, then by the choice of
$c_{\nu(p)}$ the intersection $(c_{\nu(p)}+F_n+\sum^KC_{\nu(p-1)})\cap
(F_n+\sum^KC_{\mu(q)})$ is empty but $c'+x'\in
c_{\nu(p)}+F_n+\sum^KC_{\nu(p-1)}$ and $c''+x''\in
F_n+\sum^KC_{\mu(q)}$, contradicting $c'+x'=c''+x''$. Thus
$\nu(p)=\mu(q)$. If both $p>1$ and $q>1$,
replace $c'$ with $c'-c_p'$ and $c''$ with $c''-c_q''$ and apply the
inductive hypothesis. Note that $p=1$ if and only if $q=1$ by a similar
argument implying $c'=c''$.

It follows that $c'+x'=c''+x''$ if and only if
$(c',x')=(c'',x'')$ for any $(c',x'),
(c'',x'')\in(\sum^KC_i^-)\times(F_n+\sum^KC_i)$. Hence $s_2$ is one-to-one
and the (trivial) quotient topology it induces on its image is that of
a (Cartesian) product of $\sum^KC_i^-$ and $F_n+\sum^KC_i$. Note that $F_n+\sum^KC_i$ is
Hausdorff since $F_n$ is a compact subset of the group and $\sum^KC_i$ is
finite. The Hausdorffness of $\sum^KC_i^-$ (and the continuity of
$s_1$) can be established directly 
or by a shorter indirect argument given at the end of this proof.
It remains to show that $s$, restricted to each $(C_i^-)^K\times F_n$
and $E_{j,l}\times F_n$ is continuous in the topology induced by
$s_2$. The restriction of $s$ to $(C_i^-)^K\times F_n$ is the same as
$s^+$ restricted to the same set. Since $s^+$ factors through $s_1$
and $s_2$, both of which are continuous, the continuity of $s$ on
$(C_i^-)^K\times F_n$ follows.

The continuity of $s$ restricted to $E_{j,l}\times F_n$ would follow
from the continuity of the natural additive map
$s_3:C^{K-1}\times(F_n+c_l)\to \sum^{K-1}C+F_n+c_l\subseteq\sum^KC_i^-+F_n+\sum^KC_i$. To
establish the continuity of $s_3$, observe, similar to an argument
above, that $s_3(C^{K-1}\times(F_n+c_l))\subseteq
s_3'((C_i^-)^{K-1}\times(F_n+\sum^KC_i))$ where $s_3'$ is another
addition. Just as above, $s_3'$ can be written as
$s_3'(c^1,\dots,c^{K-1},x)=s_2'(s_1'(c^1,\dots,c^{K-1}),x)$ where both
$s_1':(C_i^-)^{K-1}\to \sum^{K-1}C_i^-$ and
$s_2':(\sum^{K-1}C_i^-)\times(F_n+\sum^KC_i)\to\sum^{K-1}C_i^-+F_n+\sum^KC_i\subseteq
\sum^KC_i^-+F_n+\sum^KC_i$
are additions. Identical 
to the proof for $s_2$, the map $s_2'$ is one-to-one. Therefore, its
continuity follows from the continuity of the embedding
$\sum^{K-1}C_i^-\subseteq\sum^KC_i^-$ where the topology of each sum
is the quotient topology induced by the addition (see
below for the proof). Thus $s_3'$ is continuous and it remains to show that $s_3$ is
continuous when the topology of its range $\sum^{K-1}C+F_n+c_l$ is
inherited from the quotient topology on
$\sum^{K-1}C_i^-+F_n+\sum^KC_i$ induced by $s_3'$. This follows from the induction
hypothesis applied to the addition $s':C^{K-1}\times
F_o\to\sum^{K-1}C+F_o$ where $o\in\omega$ is large enough so that
$F_n+\sum^KC_i\subseteq F_o$. Indeed, both $s_3$ and $s_3'$ are
restrictions of $s'$, the rest is the consequence of the uniqueness of
a compact Hausdorff topology.

To show the Hausdorffness of the range of each $g:C^K\to\sum^KC$, observe, that
(viewing $G$ as a vector space over ${\mathbb F}_2$) the set
$C\setminus\{0\}$ is linearly independent. It is also easy to show
that if $C''\subseteq G'$ is an infinite convergent sequence in some boolean group
$G'$ then $C'\subseteq C''$ for some infinite linearly independent set
$C'$. Any one-to-one correspondence between $C'$ and $C\setminus\{0\}$
is now easily
seen to induce the homeomorphism between $\sum^K(C'\cup\{0\})$ as a
subspace of $G'$ and $\sum^KC$ in the corresponding quotient
topology.

Finally, the family $\set s_{i,j}(C^i\times F_j):i,j\in\omega.$
satisfies all the properties of Lemma~\ref{komega} and is easily seen
to induce a topology on $G$ that satisfies the conclusion of the lemma.
\end{proof}

\begin{remark}\label{Rgroups}
The conditions on $G$ and $D$ in Lemma~\ref{seq} can be weakened to
requiring that $D$ contain a {\it $T$-sequence\/} (see \cite{PZ1} for
the definition and  various properties and applications of
$T$-sequences) and have the property 
that each $mD$ is a closed discrete subset of $G$ (in fact an even
weaker condition would suffice).
\end{remark}

Paper \cite{TB2} mentions without proof that any non closed
subgroup of a countable $k_\omega$ group is not sequential. Below we
present a proof of a generalization of this statement to uncountable groups.

\begin{lemma}\label{SFrech}
Let $G$ be a $k_\omega$ group such that $\psi(G)=\aleph_0$ and $G$ is
not Fr\'echet. Then $G$ contains a closed copy of $S(\omega)$.
\end{lemma}
\begin{proof}
Let $\seq F_n$ be a family of compact subspaces that determines the
topology of $G$. Since $\psi(G)=\aleph_0$ each $F_n$ is first
countable and $G$ is sequential. Since $G$ is not Fr\'echet, there
exists an injective map $s:\omega^2\to G$ such that $s(i,j)\to a_i$
where $a_i\to v$ and for any $f:\omega\to\omega$ the sequence
$s(i,f(i))\not\to v$. Since each $F_n$ is first countable, the set
$\set i:s(i,j)\in F_n\hbox{ for infinitely many $j\in\omega$}.$ is
finite for every $n\in\omega$. By thinning out $s$ if necessary, one
can assume that each set of the form $\set i:s(i,j)\in F_n\hbox{ for
some $j\in\omega$}.$ is finite. If $a_i=v$ for infinitely many
$i\in\omega$, it is easy to see that there is a closed copy of
$S(\omega)$ in $\overline{s(\omega^2)}=s(\omega^2)\cup\seq
a_i\cup\{v\}$. Otherwise the same set contains a closed copy of $S_2$
so (see \cite{N}) $G$ contains a closed copy of $S(\omega)$.
\end{proof}

Let us mention without proof a result about the sequential order of
$k_\omega$-groups of countable pseudocharacter. The proof is an
extension of the techniques presented here (see also Lemma~\ref{SO} below). 
\begin{lemma}\label{so}
Let $G$ be a $k_\omega$ group such that $\psi(G)=\aleph_0$ and $G$ is
not Fr\'echet. Then $\so(G)=\omega_1$.
\end{lemma}

It turns out that a $k_\omega$ group can be `reduced' by taking an
appropriate quotient.
\begin{lemma}\label{Sresolve}
Let $G$ be a separable sequential non Fr\'echet $k_\omega$ group. Then there
exists a closed normal subgroup $H\subseteq G$ of $G$ such that $G/H$
is not Fr\'echet and $\psi(G/H)=\aleph_0$.
\end{lemma}
\begin{proof}
Suppose $G$ is not Fr\'echet. Then there exists a countable subgroup
$G'\subseteq G$ that is not Fr\'echet. Pick a countable stable family
${\cal U}$ of open subsets of $G$ such that $p$ is
one-to-one on $G'$ and $p(G'\cap F_n)=p(G')\cap p(F_n)$ for every
$n\in\omega$ where $p:G\to G/H$ is a natural quotient map and
$H=\cap{\cal U}$. Note that $\psi(G/H)=\aleph_0$. If $G/H$ is
Fr\'echet, as a sequential $k_\omega$ group it is locally compact so
let $1\in U\subseteq G/H$ be an open subset of $G/H$ such that
$\overline U$ is compact. Now $V=p^{-1}(U)\cap\overline{G'}$ is an
open neighborhood of $1$ in $\overline{G'}$. If $\overline V$ is not
(countably) compact there exists a closed infinite discrete subset
$D\subseteq\overline{V}$. Note that for every $n\in\omega$ the
intersection $F_n\cap D$ is finite and $V\cap G'$ is dense in
$\overline V$ so one can pick an infinite subset $D'\subseteq V\cap
G'$ such that every intersection $F_n\cap D'$ is finite. Hence every
$p(F_n)\cap p(D')=p(F_n\cap D')$ is finite and $p(D')$ is
infinite. Therefore $p(D')$ is an infinite closed discrete subset of
$\overline U$, a contradiction. Thus $\overline{G'}$ is locally
compact and thus Fr\'echet, contradicting the choice of $G'$.
\end{proof}

As a corollary of the lemma above and Lemma~\ref{so} one shows
\begin{proposition}\label{SO}
Let $G$ be a sequential $k_\omega$ group. Then either $G$ is locally
compact (and therefore metrizable) or $\so(G)=\omega_1$.
\end{proposition}
\begin{proof}
Note that an open map cannot raise the sequential order by Lemma~\ref{hquot}
then apply Lemmas~\ref{so} and~\ref{Sresolve}.
\end{proof}
It is possible, in fact, to obtain a closed embedding
$S_\omega\subseteq G'$ for any dense subgroup $G'$ in this
case. Lemma~\ref{so} and Proposition~\ref{SO} will
not be used in this paper.
\begin{proposition}\label{Sw}
Let $G$ be a topological group and $G'\subseteq G$ be such that
$\overline{G'}$ is a sequential non Fr\'echet $k_\omega$ group. Then
$G'$ contains a copy of $S(\omega)$ closed in $G$.
\end{proposition}
\begin{proof}
Since $G''=\overline{G'}$ is not Fr\'echet there exists a countable
$G'''\subseteq G'$ such that $\overline{G'''}$ is not Fr\'echet so we
may assume that $G'$ is countable and $G=\overline{G'}$. Using Lemma~\ref{Sresolve} 
and Lemma~\ref{SFrech}, find a closed subgoup $H\subseteq G$ such that
$\psi(G/H)=\aleph_0$, $p(G'\cap F_n)=p(G')\cap p(F_n)$ for every
$n\in\omega$, and $G/H$ is not Fr\'echet.

Suppose first there exists an injective map $s:\omega^2\to p(G')$ such
that $s(i,j)\to a_i$ for some $a_i\to v$ and $s(i,f(i))\not\to v$ for
any $f:\omega\to\omega$. Just as in Lemma~\ref{SFrech} one can thin out $s$
if necessary to assume that $s$ is a free $S_2$-embedding into
$G/H$. Let $s'(i,j)$ be the only point in $p^{-1}(s(i,j))\cap
G'$. The set $\set s'(i,j):j\in\omega.$ cannot be closed discrete in
$G$. Indeed, for some $n\in\omega$ the set $p(F_n)\cap\set
s(i,j):j\in\omega.=p(F_n\cap\set s'(i,j):j\in\omega.)$ is
infinite. Thinning $s$ again, if necessary one may assume that
$s'(i,j)\to a_i'\in p^{-1}(a_i)$.

Given $n\in\omega$ find an $i_n\in\omega$ such that $p(\cup_{i\leq
n}F_i)\cap\set s(i_n,j):j\in\omega.=\varnothing$.
Note that $G'\ni s'(i_n,j)\cdot
(s'(i_n,k))^{-1}\to s'(i_n,j)\cdot (a_{i_n}')^{-1}$ and $s'(i_n,j)\cdot
(a_{i_n}')^{-1}\to0$ and choose strictly increasing sequences
$k_i,j_i\in\omega$ such that
$s''(n,i)=s'(i_n,j_i)\cdot(s'(i_n,k_i))^{-1}\not\in\cup_{l\leq n}F_l$
and $s'':\omega^2\to G'$ is injective. Observe that $s''(i,j)\to0$ as
$j\to\infty$ and each set $\set i:s''(i,j)\in F_n\hbox{ for some
$j\in\omega$}.$ is finite. Thus $s''(\omega^2)\cup\{0\}$ is a closed
copy of $S(\omega)$ in $G'$.

If such $s$ does not exist, it follows that for every
$a\in\overline{p(G')}$ there exists a sequence $\seq c_n\subseteq
p(G')$ such that $c_n\to a$. Using Lemma~\ref{SFrech}
find a closed copy of $S(\omega)$ in $G/H$. We can assume that the map
$s:\omega^2\to G/H$ that witnesses this embedding is such that
$s(i,j)\to0$ as $j\to\infty$ and each $\set i:s(i,j)\in p(F_n)\hbox{
for some $j\in\omega$}.$ is finite. As before find an $i_n\in\omega$
such that $p(F_n)\cap\set s(i_n,j):j\in\omega.=\varnothing$. Pick a
point $c_i(j)\in p(G')\setminus p(\cup_{l\leq n}F_l)$ for each
$i,j\in\omega$ so that $c_i(j)\to s(i_n,j)$. By our assumption there
exist strictly increasing sequences $k_i,j_i\in\omega$ such that
$s'(n,i)=c_{k_i}(j_i)\to0$. By the choice of $i_n$ we may assume that
$s'(\omega^2)\cup\{0\}$ is a closed copy of $S(\omega)$ in
$p(G')$. Repeating the argument of the previous paragraph verbatim,
one can find a closed copy of $S(\omega)$ in $G'$.
\end{proof}

\begin{lemma}\label{subD}
Let $G'\subseteq G$ be a subgoup of a sequential group $G$ which is
not closed in $G$. Then $G'$ contains a closed copy of $D_\omega$.
\end{lemma}
\begin{proof}
There exists a point $a\in G\setminus G'$
and a sequence $\seq c_i\subseteq G'$ such that $c_i\to a$. Note that
$c_i\cdot c_j^{-1}\to c_i\cdot a^{-1}\not\in G'$ and $c_i\cdot
a^{-1}\to 1$. Therefore $G'\cap(\seq
c_i\cup\{a\})\cdot(\seq c_i\cup\{a\})^{-1}$ is a first countable
closed subspace of $G'$ that is not locally compact at $1$. A standard
argument produces a closed copy of $D_\omega$ in $G'$.
\end{proof}

Finally, the generalization promised before Lemma~\ref{SFrech} is a
direct corollary of Lemma~\ref{vD}, Proposition~\ref{Sw}, and Lemma~\ref{subD}.
\begin{proposition}\label{nonsequ}
Let $G$ be a sequential $k_\omega$ group and $G'\subseteq G$ be a
subgroup of $G$ such that $\overline{G'}$ is not Fr\'echet and
$\overline{G'}\not=G'$. Then $G'$ is not sequential.
\end{proposition}

Note that the condition that $\overline{G'}$ is not Fr\'echet cannot
be dropped above as the standard embeddings
$\Q\subseteq\overline{\Q}=\R\subseteq\R^\infty$ show. Here $\R^\infty$
is the direct limit of $\R^n$'s.

As a simple application of Proposition~\ref{nonsequ}, consider the
following example that indicates that sequentiality may not be readily
inherited by countable subgroups (unlike Fr\'echetness) even in the
$k_\omega$ case. Note that similar topologies have been considered
before (see, e.g.\ \cite{MPS}), although in a different context.

\begin{example}\label{Rcycl}
A $k_\omega$ sequential topology on $\R$ such that the
only proper sequential subgroups of $\R$ are closed cyclic.
\end{example}
Applying the
stronger version of Lemma~\ref{seq} mentioned in Remark~\ref{Rgroups}, one can
construct real numbers $\seq r_n$ such that $r_n\to\infty$ and a
$k_\omega$ group topology on $\R$ coarser than the original topology
and such that it is the finest group topology in which
$r_n\to0$. There is some freedom in choosing the algebraic properties
of $r_n$'s as well. Thus one can assume that
all $r_n$'s are integers, or linearly independent over $\Q$. It is
easy to see that in such a topology $\R$ becomes a sequential non
Fr\'echet group (indeed, its sequential order is $\omega_1$,
see \cite{S2}). It is also an easy observation that any countable
closed subgroup of $\R$ must be cyclic and that every cyclic subgroup
of $\R$ is dense in itself in the new topology. Thus, by
Proposition~\ref{nonsequ}, the only possible sequential subgroups of $\R$ in this
topology are cyclic and $\R$ itself.

Given any countably many cyclic subgroups of $\R$, it is not difficult
to pick a sequence as above that would make each one of the subgroups
not closed (and thus not sequential by Proposition~\ref{nonsequ}). The
question whether such a sequence can make {\it all\/}
nontrivial subgroups of $\R$ cease to be sequential seems to require
some number theoretic tools (such as deeper understanding of Kronecker
sequences $n\alpha\mod 1$ for an irrational $\alpha$) the author
currently lacks. Note that Theorem~3.1 of \cite{BDMW} implies that the
sequence as above, consisting of {\it integers\/} must have $\limsup
r_{n+1}r_n^{-1}<\infty$ for the set of $\alpha\in\R$ such that
$r_n\alpha\to0$ (topologically torsion elements) to be
countable. Since this property is not satisfied by the construction in
Lemma~\ref{seq}, a single integer sequence $\seq r_n$ is not enough to ensure that
every cyclic subgroup is not closed in this topology. It is unclear
whether non integer sequences would have this property, or if a countable
number of sequences can provide the required properties (i.e. make
every cyclic subgroup dense in $\R$).

The next lemma shows that one can add a free $S_2$-embedding that
witnesses the nonsequentiality of a given subgroup without disturbing
countably many other such embeddings. The requirement that $G$ is
boolean and co-countable can be weakened significantly at the expense
of a longer proof.

\begin{lemma}\label{fr}
Let $G$ be a boolean co-countable sequential $k_\omega$ group, $G'$ be a countable
nondiscrete subgroup of $G$. Let ${\cal S}$ be a countable family of
free $S_2$-embeddings into $G$. Let ${\cal U}$ be a countable stable
family of open subsets of $G$. Then there exists a
$k_\omega$ group topology $\tau^\infty$ on $G$, coarser than the
original topology, and a free $S_2$-embedding $s':\omega^2\to G'$ with
respect to $\tau^\infty$ such that (1) each $s\in{\cal S}$ is a free
$S_2$-embedding into $G$ with respect to $\tau^\infty$; (2) each
$U\in{\cal U}$ is open in $\tau^\infty$;
and (3) $s'(i,j)\to a_i\not\in G'$ for some $a_i\to0$ in
$\tau^\infty$.
Moreover $\tau^\infty$ is the finest group topology on $G$ coarser
than its original topology and such that $S_i\to a_i$ for some $\seq S_i$.
\end{lemma}
\begin{proof}

Suppose the topology of $G$ is determined by a countable family $\seq
F_n$ of compact subsets closed with respect to finite intersections,
finite unions, and natural algebraic operations. For the sake of
simplicity, assume that ${\cal S}=\{s\}$ where $s:\omega^2\to G$ is a
free $S_2$-embedding. The case of a countable ${\cal S}$ is similar.
Let $G'$ be an arbitrary countable subgroup. Use Lemma~\ref{stabres}
to extend $\cal U$ to a countable family
$\cal U'$ of open subsets of $G$ such that $\cal U'$ is stable,
$H=\cap\cal U'$ is a closed subgroup of $G$, $p$ is one-to-one on
$G'\cup\overline{s(\omega^2)}$ and for any $n\in\omega$ the set $\set
i\in\omega: p\circ s(i,j)\in p(F_n)\hbox{ for some $j\in\omega$}.$ is
finite. Now the family $\langle p(F_n): n\in\omega\rangle$ of
compact subsets determines the topology of the countable group $G/H$,
so $p\circ s$ is a free $S_2$-embedding into $G/H$, and $U=p^{-1}(p(U))$ for each
$U\in\cal U'$. We may assume, by extending $\cal U'$ if necessary, that
$\set p(U): U\in{\cal U}'.$ forms a basis of open neighborhoods for
some first-countable group topology $\tau_\omega$ coarser that the
topology $\tau$ induced by $p$.

Find a function $f:\omega\to\omega$ such that $K=\overline{p\circ s(\set (m,n):
n\geq f(n).)}$ is compact in $\tau_\omega$. Since the set $D=\set
p\circ s(m,n): n<f(m).$ is closed and discrete in $\tau$  we can
extend $\tau_\omega$ to a first countable topology $\tau_\omega'$
coarser than $\tau$ so that $D$ is closed and discrete in
$\tau_\omega'$.

If $G'$ is not discrete, its image $p(G')$ is dense in
itself in $\tau$, $\tau_\omega$, and $\tau_\omega'$. Therefore, one
can find a countable family $\seq V_n$ of (relatively) open 
in $\tau_\omega'$ subsets of $p(G')$ such that $\seq V_n\to0$ in
$\tau_\omega'$. Note that if $C=\seq c_n\cup\{a\}$ is chosen so that
$a\in p^{-1}(0)$ and $c_n\in p^{-1}(V_n)$ then for any $n,k\in\omega$
the intersection $(F_n+\sum^kC)\cap D'$ is finite, where $D'=\set
s(m,n):n<f(m).$. Indeed $p(C)$
is a compact subset of $p(G')$ in $\tau_\omega'$, $p$ is one-to-one on
$D'$, and $p(D')=D$ is a closed discrete subset of $G/H$ in $\tau_\omega'$.

Use induction on $i\in\omega$ to find $c_i'\in V_i$ such that for any
$m>i$
$$
(c_m'+\cup_{j\leq i}p(F_j)+\sum^iC_i)\cap (K\cup(\cup_{j\leq i}p(F_j)))=\varnothing
$$
where $C_i=\set c_j':j\leq i.\cup\{0\}$. Since each $V_n$ is dense in
itself (in both $\tau_\omega'$ and $\tau_\omega$), while $K$, $p(F_j)$'s
and their sums
are compact in $\tau_\omega$ and thus scattered, such a choice of $c_i$
is possible.

Now inductively assume that for any $n\in\omega$ and any $m<M$ the set
$$
\set j:p\circ s(j,l)\in p(F_n)+\sum^mC\hbox{ for some $l\geq f(j)$}.
$$ is finite. Pick $i=\max\{M,n\}$. Note that the set $p(F_n)+\sum^MC$
is a union of $p(F_n)+\sum^MC_i^-$ and finitely many sets of the form
$p(F_n)+\sum^{M-1}C+c_l'$ for some $l\leq i$, where $C_i^-=\set
c_j':j>i.\cup\{0\}$. By the induction hypothesis it is enough to show
that the set
$$
\set j:p\circ s(j,l)\in
p(F_n)+\sum^MC_i^-\hbox{ for some $l\geq f(j)$}.
$$
is finite. To see
this, observe that for $c\in\sum^MC_i^-$ either
$c=0$ or one can write $c=c^1+\cdots+c^p$ where
$c^j=c_{\nu(j)}'$ such that $\nu$ is strictly increasing and
$\nu(j)>i$. In the latter case, by the choice of $c_i'$ the
intersection $(c_{\nu(p)}'+p(F_n)+\sum^{\nu(p)-1}C_{\nu(p)-1}^-)\cap
K=\varnothing$ so
\begin{gather*}
\set j:p\circ s(j,l)\in p(F_n)+\sum^MC_i^-\hbox{ for some $l\geq f(j)$}.\\
=\set j:p\circ s(j,l)\in p(F_n)\hbox{ for some $l\geq f(j)$}..
\end{gather*}
The last set is finite since $p(F_n)$ is compact and
$p\circ s$ is a free $S_2$-embedding in $\tau$.

Let $c_i\in G'\cap p^{-1}(c_i')$, $S=\seq c_i$. It is immediate that
$S$ is a closed discrete subset of $G$ such that for every
$n,k\in\omega$ the set
$$
\set j:s(j,l)\in F_n+\sum^kS\hbox{ for some
$l\in\omega$}.
$$
is finite. Moreover, since $p(S)\to0$ in $\tau_\omega$, $S\ain U$ for
every $U\in{\cal U}'$.

Choose an infinite sequence $\seq a_i\subseteq H$ such that
$a_i\to0$. Use Lemma~\ref{seq} to find an $S_0\subseteq S$ such that $S\to
a_0$ in a $k_\omega$-topology $\tau'$ on $G$ coarser than the
original topology and such that $\tau'$ is the finest group topology
on $G$ with these properties. Note that with respect to $\tau'$ both
${\cal U}$ and ${\cal U}'$ remain stable
countable families of open subsets of $G$. Thus the
argument above can be repeated using $\tau'$ instead of the original
topology on $G$ to obtain $S_1\to a_1$ and $\tau''$. Eventually one
can construct $k_\omega$-topologies $\tau^{(i)}$ and sequences $S_i\to
a_i$ in $\tau^{(i)}$ such that (i) ${\cal U}$ is a countable stable
family of subsets of $G$ open in 
$\tau^\infty=\cap_i\tau^{(i)}$; (ii) $\tau^{(i+1)}$ is the finest group
topology on $G$ coarser than $\tau^{(i)}$ in which $S_i\to a_i$; (iii)
each $\tau^{(i)}$ is $k_\omega$; (iv) $S_i$ is a closed discrete subset
of $G$ in $\tau^{(j)}$ for all $i\geq j$; (v) $s$ is a free $S_2$-embedding in
$\tau^\infty$.

Let now $\seq K_i$ be a family of compact subsets of $G$ in
$\tau^\infty$ that determines $\tau^\infty$. Let $S_i=\seq c^i_j$ and
put $s'(j,l)=c^j_l$. We can assume that $s'$ is one-to-one and (using
(iv) above) that each $K_n$ intersects at most finitely many
$S_i$'s. Thus $s'$ is a free $S_2$-embedding into $G$ in
$\tau^\infty$.
\end{proof}

\begin{remark}\label{otherStwo}
As an alternative to building $S_i$'s at the end of the proof above
one could show that after adding $S_0$ the closure of $G'$ in $G$ will
not be Fr\'echet and then use the proof of Proposition~\ref{nonsequ}
that produces a closed copy of $S_2^-$ in $G'$.
\end{remark}

The next lemma is used to make the quotients of $G$ Fr\'echet.
\begin{lemma}\label{vertica}
Let $G$ be a boolean co-countable $k_\omega$ group and let ${\cal U}$ be a
countable stable family of open subsets of $G$. Let
$H=\cap{\cal U}$ be a closed subgroup of $G$ such that $H\not\subseteq
F+E$ for any compact
$F\subseteq G$ and any countable subset $E\subseteq G$. Let
$\sigma:\omega^2\to G/H$ be a free $S_2$-embedding in the topology
induced by the natural qoutient map $p$ and let ${\cal S}$ be a 
countable family of free $S_2$-embeddings into $G$. Then there exists
a countable subset $S\subseteq G$ such that each $U\in{\cal U}$ is
open in the topology $\tau$ which is the finest topology coarser than
the original topology of $G$ and such that $S\to v$ for some $v\in G$,
each element of ${\cal S}$ is a free $S_2$-embedding in $\tau$ and
$\sigma$ is no longer a free $S_2$-embedding.
\end{lemma}
\begin{proof}
To simplify notation, assume ${\cal S}=\{s\}$ for some free
$S_2$-embedding $s:\omega^2\to G$. Let $\seq F_n$ be a countable
family of compact subsets of $G$ that determines the topology of $G$.
Since $\sigma$ is a free $S_2$-embedding into $G/H$, one can
find a function $f:\omega\to\omega$ such that the sequence $\langle
\sigma(n,f(n)):n\in\omega\rangle$ converges to the vertex of
$\sigma$ in the first-countable group topology $\tau_\omega$ generated
by the base of open neighborhoods $\set p(U):U\in{\cal U}.$. By
induction on $i\in\omega$ find a sequence $c_i\in
p^{-1}(\sigma(i,f(i)))$ such that for any $n\in\omega$, $m>i$
$$
(c_m+\cup_{j\leq i}F_j+\sum^iC_i)\cap(K\cup(\cup_{j\leq i}F_j))=\varnothing
$$
where $C_i=\set c_j:j\leq i.$. Since $K$ is countable, the condition
imposed on the kernel of $p$ makes such selection possible.

Just as in Lemma~\ref{fr} one can show that $S=\seq c_n$ has the
property that for any $k,n\in\omega$ the set $\set
j:s(j,l)\in\sum^kS+F_n.$ is finite and $S\ain U$ for each $U\in{\cal
U}$. It is also immediate that $S$ is a closed discrete subset of
$G$. Apply Lemma~\ref{seq} to find a subset $S'\subseteq S$ such that
$S'\to v\in p^{-1}(v_\sigma)$ where $v_\sigma$ is the vertex of
$\sigma$.
\end{proof}

\section{Example}\label{exa}
This is the main example in the paper.
\begin{example}[$\diamondsuit$]\label{nrg}
A sequential group $G$ such that every countable sequential subgroup of $G$ is
discrete and every quotient of $G$ is either Fr\'echet or has an
uncountable pseudocharacter.
\end{example}
Let $A\subseteq 2^{\omega_1}$ be a subspace homeomorphic to the one
point compactification of a discrete space of size $\omega_1$. We can
assume, by thinning out and translating $A$ if necessary, that $0$ is
the only non isolated point of $A$ and that all the isolated points
are linearly independent over ${\mathbb F}_2$. Let $G$
be the subgroup of $2^{\omega_1}$ generated by $A$.

\begin{claim}\label{csh}Let $n\in\omega$, $E\subseteq G$ be a
countable subset of $G$. Let $H$ be a subgroup of $G$ generated by an
uncountable subset $B$ of $A$. Then $\sum^nA+E$ does not cover $H$. Indeed,
let $E'\subset A$ be a countable set whose span contains $E$, and
let $a^1,\ldots,a^{n+1}\in B$ be points of $B$ that are not in the
span of $E'$. Then $a^1+\cdots+a^{n+1}\not\in\sum^nA+E$.
\end{claim}

Finally, let
$\tau_0$ be the finest group topology on $G$ that induces the original
topology on $A$. By an observation after Lemma~\ref{komega} $G$ is a
co-countable $k_\omega$-group.

Given countable subsets $S_i\subseteq G$, $0\leq i\leq k$, and $n\in\omega$, consider
the natural additive map
$a_{\{S_i\},k,n}:\prod_{i=0}^k(S_i\cup\{0\})\times\prod^nA\to(\sum_{i=0}^k(S_i\cup\{0\}))+\sum^nA$,
viewed as a quotient map where $S_i\cup\{0\}$ is given the unique
topology such that $S_i\to0$. Suppose $a_{\{S_i\},k,n}$ induces a (compact)
Hausdorff topology on its image. Then the image has weight
$\omega_1$ so let $\set U^\alpha_{\{S_i\},k,n}:\alpha\in\omega_1.$ be a base
of open neighborhoods for the image. Now use CH to find
$B:\omega_1\to 2^G$ that lists every $U^\alpha_{\{S_i\},k,n}$ with the properties
above. Fix a $\diamondsuit$-sequence $\set
A_\alpha:\alpha\in\omega_1.$. Let $\set G_\alpha:\alpha\in\omega_1.$,
$\set \sigma_\alpha:\alpha\in\omega_1.$
list every countable subgroup of $G$ and every one-to-one map
$\sigma:\omega^2\to G$ respectively, unboundedly many times.

Construct, by induction on $\alpha\in\omega_1$, decreasing $k_\omega$ group
topologies $\tau_\alpha$ on $G$, one-to-one maps $s_\alpha:\omega^2\to
G_\alpha$, increasing countable families ${\cal S}_\alpha$ of countable subsets
of $G$, and increasing countable families ${\cal U}_\alpha$ of open in
$\tau_\alpha$ subsets of $G$ such that
\begin{itemize}
\item[(a)]each $\tau_\alpha$ is the finest group topology on $G$ that
induces the original topology on $A$ such that $S\to0$ for each
$S\in{\cal S}_\alpha$.

\item[(b)]if $G_\alpha$ is not closed discrete in $\tau_\alpha$ then
$s_\alpha$ is a free $S_2$-embedding into $G$ with respect to
$\tau_{\alpha+1}$ such that $s_\alpha(i,j)\to a_i\not\in G_\alpha$
where $a_i\to0$.

\item[(c)]each $s_\beta$ is a free $S_2$-embedding in $\tau_\alpha$ for
$\beta<\alpha$.

\item[(d)]each ${\cal U}_\alpha$ is stable.

\item[(e)]if $V=\cup\set B(\beta):\beta\in A_\alpha.$ is such that
$0\in\Int(V)$ in $\tau_\alpha$ then there is $U\in{\cal U}_\alpha$
such that $0\in U\subseteq \Int(V)$.

\item[(f)]$p_\alpha\circ\sigma_\alpha$ is not a free $S_2$-embedding in $G/H_\alpha$
with respect to the topology induced by $p_\alpha:G\to G/H_\alpha$ where
$H_\alpha=\cap{\cal U}_\alpha$ and the topology on $G$ is $\tau_\alpha$.

\end{itemize}
Suppose $\tau_\beta$, $s_\beta$, ${\cal U}_\beta$, and ${\cal
S}_\beta$ satisfying (a)--(f) have been built for $\beta<\alpha$. Put
$\tau_\alpha'=\cap_{\beta<\alpha}\tau_\beta$, ${\cal
U}_\alpha'=\cup_{\beta<\alpha}{\cal U}_\beta$, and ${\cal
S}_\alpha'=\cup_{\beta<\alpha}{\cal S}_\beta$. Observe that each
$U\in{\cal U}_\alpha'$ is open in $\tau_\alpha'$ and that
$\tau_\alpha'$ is the finest topology such that it induces the
original topology on $A$ and $S\to0$ for each $S\in{\cal S}_\alpha'$.
If $V=\cup\set
B(\beta):\beta\in A_\alpha.$ has a nonempty interior in $\tau_\alpha'$
and $0\in \Int(V)$
extend ${\cal U}_\alpha'$ to a countable family ${\cal U}_\alpha$ of
open subsets of $G$ in $\tau_\alpha'$ such that ${\cal U}_\alpha$ is
stable and $0\in U\subseteq \Int(V)$ for some $U\in{\cal
U}_\alpha$. If $G_\alpha$ is not a closed discrete
subgroup of $G$ in $\tau_\alpha'$, use Lemma~\ref{fr} to find a countable family ${\cal
S}$ of countable subsets of $G$ and a one-to-one map
$s_\alpha:\omega^2\to G_\alpha$ such that $\set s_\beta:\beta\leq\alpha.$ are free
$S_2$-embeddings with respect to $\tau_\alpha''$, and each $U\in{\cal
U}_\alpha$ is open in $\tau_\alpha''$, where $\tau_\alpha''$ is
the finest group topology on $G$ such that $S\to0$ for each $S\in{\cal
S}$, coarser than $\tau_\alpha'$. Otherwise put
$s_\alpha=s_\beta$ for some $\beta<\alpha$, ${\cal S}=\varnothing$,
and $\tau_\alpha''=\tau_\alpha'$.

Now if $p\circ\sigma_\alpha$ is a free $S_2$-embedding into $G/H$
where $p:G\to G/H$ is the natural quotient map in $\tau_\alpha''$, and
$H=\cap{\cal U}_\alpha$, use Lemma~\ref{vertica} to find a convergent
sequence $S\subseteq G$ in a $k_\omega$-topology $\tau_\alpha$ coarser
than $\tau_\alpha''$ and such that each $U\in{\cal U}_\alpha$ is open
in $\tau_\alpha$, and $p\circ\sigma_\alpha$ is not a free $S_2$-embedding in
$\tau_\alpha/H$.

That the conditions of Lemma~\ref{vertica} are satisfied follows from Claim~\ref{csh}
above and an easy observation that each compact subset of $G$ in any
$\tau_\alpha$ is included in a union of countably many translations of
a single sum of $n$ copies of $A$, as follows from (a).

If $p\circ\sigma_\alpha$ is not a free $S_2$-embedding take $S$ to be
an arbitrary convergent sequence in $\tau_\alpha''$, put
$\tau_\alpha=\tau_\alpha''$.

Put ${\cal S}_\alpha={\cal S}_\alpha'\cup{\cal S}\cup\{S\}$. (a)--(f) follow.

Let $\tau=\cap_{\alpha<\omega_1}\tau_\alpha$. It is immediate that
$\tau$ is a sequential topology on $G$ invariant with respect to
translations. Properties (a)--(f) imply that each $s_\alpha$ is a free
$S_2$-embedding with respect to $\tau$. Since each $G_\alpha$ is
listed unboundedly many times, if $G_\alpha$ is not closed discrete in
$\tau$ it is not closed discrete in some $\tau_\beta$ such that
$G_\alpha=G_\beta$. Then (b) ensures that the topology inherited by
$G_\alpha$ from $\tau$ is not sequential.

To show that $\tau$ is a group topology put ${\cal U}=\cup_\alpha{\cal
U}_\alpha$. Since each $U\in{\cal U}$ is open in $\tau$ ($\{{\cal
U}_\alpha\}$ is increasing and each $U\in{\cal U}_\alpha$ is open in
$\tau_\alpha$) and ${\cal U}$ is stable it is enough to
show that ${\cal U}_\alpha$ is a basis 
of open neighborhoods of $0\in G$ in $\tau$.

Let $F\subseteq G$ be a subset of $G$ closed in $\tau$ such that
$0\not\in F$.
Let $\alpha<\omega_1$ and $O\subseteq\omega_1$ be the set of all
$\beta\in\omega_1$ such that $B(\beta)$ is an open subset of some compact $K\subseteq G$
in $\tau_{t(\beta)}$ for $t(\beta)>\beta$ and $\overline{B(\beta)}\cap
F=\varnothing$ where the closure is taken in $\tau_{t(\beta)}$.
Put $O_\alpha=O\cap\alpha$ and call $\alpha$ {\it $F$-saturated in
$\tau$} if the following conditions are met:
\begin{itemize}
\item[(1)]if $\beta_1,\ldots,\beta_n\in O_\alpha$ and $K\subseteq G$
is compact (in $\tau_\alpha$) then there exist $\beta^1,\ldots,\beta^k\in
O_\alpha$ such that each $B(\beta^i)\cap K$ is relatively open in $K$, and
$\overline{\cup_{i\leq n}B(\beta_i)}\subseteq\cup_{j\leq
k}B(\beta^j)\subseteq\overline{\cup_{j\leq k}B(\beta^j)}\subseteq
G\setminus F$.

\item[(2)]for each $\beta\in O_\alpha$ $t(\beta)<\alpha$ and there
exists $\beta\in O_\alpha$ such that $0\in B(\beta)$.
\end{itemize}
Note that an inductive construction similar to that of Lemma~\ref{komega} shows
that $0\in\Int_{\tau_\alpha}(\cup_{\beta\in O_\alpha}B(\beta))\subseteq G\setminus F$ and a standard argument
shows that the ordinals $F$-saturated in $\tau$ form a club. Therefore
there exists an $\alpha\in\omega_1$ such that $A_\alpha=O_\alpha$ and
$0\in\Int(A_\alpha)\subseteq F$ in $\tau_\alpha$ so by (e) above there
is a $U\in{\cal U}_\alpha\subseteq{\cal U}$ such that $0\in U\subseteq
G\setminus F$.

Suppose $H\subseteq G$ is a closed (in $\tau$) subgroup of $F$ such
that $G/H$ is not Fr\'echet and $\psi(G/H)=\omega$. Since $G$ is
co-countable, $G/H$ is a countable group so there exists a free
$S_2$-embedding $\sigma:\omega^2\to G/H$ such that
$\sigma=p\circ\sigma'$ where $\sigma':\omega^2\to G$ is a free
$S_2$-embedding in $\tau$. Since ${\cal U}$
forms a base of open neighborhoods of $0\in G$ in $\tau$ one can find
$\alpha\in\omega_1$ such that $H_\beta\subseteq H$
for any $\beta\geq\alpha$ and thus $p_\beta\circ\sigma'$ is a free
$S_2$-embedding. If $\beta$ is large enough $\sigma'=\sigma_\beta$ and thus by (f)
$p_\beta\circ\sigma_\beta$ is not a free $S_2$-embedding contradicting the
choice of $\sigma'$.

\begin{remark}\label{GG}
It is an easy observation that by replacing (e) above with
\begin{itemize}
\item[(e')]if $V=\cup\set B^2(\beta):\beta\in A_\alpha.$ has a nonempty
interior in $\tau_\alpha^2$ and $(0,0)\in\Int(V)$ then there is $U\in{\cal U}_\alpha$ such
that $(0,0)\in U^2\subseteq \Int(V)$.

\end{itemize}
and adjusting the definition of a saturated ordinal to
\begin{itemize}
\item[(1)]if $\beta_1,\ldots,\beta_n\in O_\alpha$ and $K\subseteq
G^2$ is compact (in $\tau_\alpha^2$) then there exist $\beta^1,\ldots,\beta^k\in
O_\alpha$ such that each $B^2(\beta^i)\cap K$ is relatively open in $K$, and
$\overline{\cup_{i\leq n}B^2(\beta_i)}\subseteq\cup_{j\leq
k}B^2(\beta^j)\subseteq\overline{\cup_{j\leq k}B^2(\beta^j)}\subseteq
G^2\setminus F$.

\item[(2)]for each $\beta\in O_\alpha$ $t(\beta)<\alpha$ and there
exists $\beta\in O_\alpha$ such that $0\in B(\beta)$.
\end{itemize}
after assuming $F\subseteq G^2$ and defining $O_\alpha$ and $t(\beta)$ appropriately, one can
construct a group $G$ as above with the additional property that $G^2$
is sequential (indeed, all finite powers of $G$ can be made sequential
after a minor change to the method above). By `trapping' closed copies of
$S(\omega)$ in each $G_\alpha$ using Lemma~\ref{Sw} (or noting that such a copy must exist
in any sequential group that contains $G_\alpha$) and using
Lemmas~\ref{vD} and \ref{subD} one can construct a $G$ as above with
the additional property that the only sequential subgroups of $G$ are
closed and uncountable (in fact, the $G$ constructed above has this
property automatically, as noted above). This, in turn, implies that
every sequential subgroup of $G$ is either countable and discrete or
contains a compact subspace of uncountable pseudocharacter.

Finally, a more precise statement of Lemma~\ref{fr} (with a modified
proof) would allow a construction of such $G$ with $\so(G)=\omega+1$.
\end{remark}
The results above leave open a number of interesting questions.
\begin{question}\label{Group}
Does there exist (consistently or in ZFC) a sequential group $G$ such
that all countable sequential subgroups of $G$ are finite and all the
quotients of $G$ are either first countable or have uncountable
pseudocharacter?
\end{question}
Note that the construction of Example~\ref{nrg} produces a closed copy of
$S(\omega)$ in $G$. If $G\times G$ is also sequential then no
nontrivial quotient of $G$ can be first countable. Indeed, otherwise
some open homeomorphic image of $G\times G$ would contain closed
copies of both $S(\omega)$ and $D_\omega$ contradicting Lemma~\ref{vD}.

A very strong version of the question above is also open. Note that it
becomes meaningful only when the negation of CH is assumed.
\begin{question}\label{thegroup}
Does there exist a sequential group $G$ such that every sequential
subgroup of $G$ is either finite or closed, uncountable, and of
countable index and every quotient of $G$ is either first countable or
has a pseudocharacter $>\omega_1$? Can such $G$ (if exists) have any
sequential order?
\end{question}
Finally, if the quotient requirements are dropped, can the group be
made $k_\omega$?
\begin{question}Does there exist a sequential $k_\omega$ group $G$
such that every countable sequential subgroup of $G$ is discrete
(finite)? In particular does there exist a sequence as described in
Example~\ref{Rcycl} that makes $\R$ into such a group? Countably many
sequences?
\end{question}

\end{document}